\documentclass[a4paper]{article}

\usepackage{amsmath,amssymb}
\usepackage{amscd}
\usepackage{mathrsfs}
\usepackage{amsthm}
\usepackage[all]{xy}

\theoremstyle{plain}

\newtheorem{theorem}{\bf Theorem}[section]
\newtheorem{lemma}[theorem]{\bf Lemma}

\newtheorem{corollary}[theorem]{\bf Corollary}

\theoremstyle{definition}
\newtheorem{definition}[theorem]{\bf Definition}

\newtheorem{remark}[theorem]{\bf Remark}

\newcommand{\eqa}[1]{
\begin{align*}
#1
\end{align*}}

\newcommand{\nai}[2]{\langle #1,#2\rangle}

\title{Ultraproducts, QWEP von Neumann Algebras,
and the 
Effros-Mar\'echal Topology}

\author{Hiroshi ANDO \and Uffe HAAGERUP \and Carl WINSL\O W}

\begin{document}
\maketitle
\begin{abstract}
Based on the analysis on the Ocneanu/Groh-Raynaud ultraproducts and the Effros-Mar\'echal topology on the space vN$(H)$ of von Neumann algebras acting on a separable Hilbert space $H$, we show that for a von Neumann algebra $M\in \text{vN}(H)$, the following conditions are equivalent: (1) $M$ has the Kirhcberg's quotient weak expectation property (QWEP). (2) $M$ is in the closure of the set $\mathcal{F}_{\rm{inj}}$ of injective factors on $H$ with respect to the Effros-Mar\'echal topology. (3) $M$ admits an embedding $i$ into the Ocneanu ultrapower $R_{\infty}^{\omega}$ of the injective III$_1$ factor $R_{\infty}$ with a normal faithful conditional expectation $\varepsilon\colon R_{\infty}^{\omega}\to i(M)$. (4) For every $\varepsilon>0, n\in \mathbb{N}$ and $\xi_1,\cdots,\xi_n\in \mathcal{P}_M^{\natural}$, there is $k\in \mathbb{N}$ and $a_1,\cdots,a_n\in M_k(\mathbb{C})_+$, such that $|\nai{\xi_i}{\xi_j}-\text{tr}_k(a_ia_j)|<\varepsilon\ (1\le i,j\le n)$ holds, where tr$_k$ is the tracial state on $M_k(\mathbb{C})$, and $\mathcal{P}_M^{\natural}$ is the natural cone in the standard form of $M$.
\end{abstract}

\noindent
%{\bf Keywords}. Hilbert space, dense subspace. 

\medskip

\noindent
%{\bf Mathematics Subject Classification (2000)}. 

\medskip

\section{Introduction}
In the seminal paper \cite{Kirchberg}, Kirchberg revealed the unexpected connection among tensor products of C$^*$-algebras, the weak expectation property (WEP) of Lance \cite{Lance}, and the Connes' embedding conjecture for type II$_1$ factors. A C$^*$-algebra $A$ is said to have WEP, if for any faithful representation $A\subset \mathbb{B}(H)$, there is a unital completely positive (ucp) map $\Phi\colon \mathbb{B}(H)\to A^{**}$ such that $\Phi|_A=\text{id}_A$. A C$^*$-algebra $A$ is said to have the quotient weak expectation property (QWEP) if there is a surjective *-homomorphism from a C$^*$-algebra $B$ with WEP onto $A$. Among other interesting results, Kirchberg proved that the following conditions are equivalent:
\begin{itemize}
\item[(1)] $C^*(\mathbb{F}_{\infty})\otimes_{\text{min}}C^*(\mathbb{F}_{\infty})=C^*(\mathbb{F}_{\infty})\otimes_{\text{max}}C^*(\mathbb{F}_{\infty})$.
\item[(2)] Every C$^*$-algebra has QWEP.
\item[(3)] Every II$_1$ factor $N$ with separable predual embeds into the tracial ultrapower $R^{\omega}$ (for some fixed free ultrafilter $\omega$ on $\mathbb{N})$ of the hyperfinite II$_1$ factor $R$.  
\end{itemize}
(2) is now called Kirchberg's QWEP conjecture, and (3) is called the Connes' embedding conjecture, which was mentioned for the first time in \cite{Connes76}. See also the excellent survey \cite{Ozawa} on the QWEP conjecture. Nowadays many interesting equivalent conditions of the above conjectures are known.  
On the other hand, the QWEP conjecture also has another connection to the topological properties of the space of von Neumann algebras. The Effros-Mar\'echal topology on the space vN$(H)$ of von Neumann algebras acting on a fixed Hilbert space $H$ is the weakest topology on vN$(H)$ for which the map 
\[N\mapsto \|\varphi|_N\|\]
is continuous for every $\varphi\in \mathbb{B}(H)_*$. Based on the work of Effros \cite{Effros1,Effros2} on the Borel structure of vN$(H)$, this topology was defined by Mar\'echal in \cite{Marechal} so that it generates the Effros Borel structure. Later it was intensively studied in \cite{HW1,HW2}, where the second and third-named authors showed that this topology was indeed well-matched with Tomita-Takesaki theory and could be used as a tool for the study of global properties of von Neumann algebras. Among other things, it was proved \cite[Theorem 5.8]{HW2} that when $H$ is separable, a II$_1$ factor $N\in {\rm{vN}}(H)$ is in the closure of the set $\mathcal{F}_{\rm{inj}}$ of injective factors on $H$, if and only if $N$ embeds into $R^{\omega}$. Consequently, Connes' embedding conjecture (hence all conditions (1)-(3) above) is equivalent to
\begin{itemize}
\item[(4)] $\mathcal{F}_{\rm{inj}}$ is dense in $\text{vN}(H)$.
\end{itemize}
In this paper, we establish further connection among ultraproducts, the approximation by injective factors, and QWEP von Neumann algbebras. To this we make use of the recent work of the first and the second-named authors on the ultraproducts of general von Neumann algebras \cite{AndoHaagerup}. We also carry on the analysis of the natural cone in the standard form \cite{Haagerup1}. The main result of the paper is as follows. Let $R_{\infty}$ (resp. $R_{\lambda}$) denote the injective factor of type III$_1$ (resp. type III$_{\lambda}$\ $(0<\lambda<1)$). We assume that $H$ is separable and infinite-dimensional.
\begin{theorem}
Let $0<\lambda<1$ and let $M\in {\rm{vN}}(H)$, and let $\mathcal{P}_M^{\natural}$ be the natural cone in the standard form of $M$. The following conditions are equivalent.
\begin{itemize}
\item[{\rm{(1)}}] $M$ has {\rm{QWEP}}.
\item[{\rm{(2)}}] $M\in \overline{\mathcal{F}_{\rm{inj}}}$. 
\item[{\rm{(3)}}] There is an embedding $i\colon M\to R_{\infty}^{\omega}$ and a normal faithful conditional expectation $\varepsilon\colon R_{\infty}^{\omega}\to i(M)$.
\item[{\rm{(4)}}] There is an embedding $i\colon M\to R_{\lambda}^{\omega}$ and a normal faithful conditional expectation $\varepsilon\colon R_{\lambda}^{\omega}\to i(M)$.
\item[{\rm{(5)}}] There is $\{k_n\}_{n=1}^{\infty}\subset \mathbb{N}$ and a normal faithful state $\varphi_n$ on $M_{k_n}(\mathbb{C})\ (n\in \mathbb{N})$, an embedding $i\colon M\to (M_{k_n}(\mathbb{C}),\varphi_n)^{\omega}$ and a normal faithful conditional expectation $\varepsilon\colon  (M_{k_n}(\mathbb{C}),\varphi_n)^{\omega}\to i(M)$.
\item[{\rm{(6)}}] For every $\varepsilon>0, n\in \mathbb{N}$ and $\xi_1,\cdots,\xi_n\in \mathcal{P}_M^{\natural}$, there exist $k\in \mathbb{N}$ and $a_1,\cdots,a_n\in M_k(\mathbb{C})_+$ such that 
\[|\nai{\xi_i}{\xi_j}-{\rm{tr}}_k(a_ia_j)|<\varepsilon\ \ \ \ \ (1\le i,j\le n).\]
\end{itemize}
\end{theorem}     
Here, $(N_n,\varphi_n)^{\omega}$ denotes the Ocneanu ultraproduct of $\{N_n\}_{n=1}^{\infty}$ with respect to the sequence of normal faithful states $\{\varphi_n\}_{n=1}^{\infty}$ (see $\S2$), and the embedding is understood to be a unital normal injective *-homomorphism.\\
The organization of the paper is as follows. In $\S2$ we recall necessary backgrounds about ultraproducts, Effros-Mar\'echal topology. In $\S$3, we prove the relation between embedding into the Ocneanu ultraproduct and the limit in the Effros-Mar\'echal topology. Using this, we prove the equivalence of (1)-(5) in the above theorem. In $\S$4, we prove the equivalence of (6) and (1) in the above theorem.  
\section{Preliminaries}
\subsection{Notation}
Throughout the paper, $\omega \in \beta \mathbb{N}\setminus \mathbb{N}$ denotes a fixed free ultrafilter on $\mathbb{N}$. $H$ denotes a separable Hilbert space and ${\rm{vN}}(H)$ the space of all von Neumann algebras acting on $H$ (we assumed all von Neumann algebras contain $1={\rm{id}}_H$). Let $M$ be a von Neumann algebra, $\varphi$ be a normal state on $M$. As usual, we define two (semi-)norms $\|\cdot \|_{\varphi},\ \|\cdot \|_{\varphi}^{\sharp}$ by
\[\|x\|_{\varphi}:=\varphi(x^*x)^{\frac{1}{2}},\ \|x\|_{\varphi}^{\sharp}:=\varphi \left (x^*x+xx^*\right )^{\frac{1}{2}},\ x\in M.\]
For a sequence $(M_n)_n$ of von Neumann algebras, $\ell^{\infty}(\mathbb{N},M_n)$ is the C$^*$-algebra of all norm-bounded seuquences $(x_n)_n\in \prod_{n\in \mathbb{N}}M_n$.
\subsection{Natural cone $\mathcal{P}_M^{\natural}$} 
 Recall the definition of the standard form.
\begin{definition}\label{def: standard form}
Let $(M,H,J,\mathcal{P}_M^{\natural})$ be a quadruple, where $M$ is a von Neumann algebra, $H$ is a Hilbert space on which $M$ acts, $J$ is an antilinear isometry on $H$ with $J^2=1$, and $\mathcal{P}_M^{\natural}\subset H$ is a closed convex cone which is self-dual, i.e., $\mathcal{P}_M^{\natural}=(\mathcal{P}_{M}^{\natural})^0$, where
\[(\mathcal{P}_M^{\natural})^0:=\{\xi \in H;\ \nai{\xi}{\eta}\ge 0,\ \eta\in \mathcal{P}_M^{\natural}\}.\]
Then $(M,H,J,\mathcal{P}_M^{\natural})$ is called a {\it standard form} if the following conditions are satisfied:
\begin{itemize}
\item[1.] $JMJ=M'$.
\item[2.] $J\xi=\xi,\ \xi\in \mathcal{P}_M^{\natural}$.
\item[3.] $xJxJ(\mathcal{P}_M^{\natural})\subset \mathcal{P}_M^{\natural},\ x\in M$.
\item[4.] $JxJ=x^*,\ x\in \mathcal{Z}(M)$.
\end{itemize}
\end{definition}
We remark that condition 4. automatically follows from the other three conditions \cite[Lemma 3.19]{AndoHaagerup}. The existence and the uniqueness of the standard form was established in \cite{Haagerup1}. $\mathcal{P}^{\natural}_M$ is called the {\it natural cone} of $M$. It was independently introduced by Connes \cite{Connes73} and Araki \cite{Araki}, and if $M$ is $\sigma$-finite with a normal faithful state $\varphi$, then on the GNS Hilbert space $H=L^2(M,\varphi)$, $\mathcal{P}_M^{\natural}$ is realized as
\[\mathcal{P}_M^{\natural}=\overline{\{aJ_{\varphi}aJ_{\varphi}\xi_{\varphi};a\in M\}}=\overline{\left \{\Delta_{\varphi}^{\frac{1}{4}}a\xi_{\varphi};a\in M_+\right \}}.\]
Here, $\Delta_{\varphi}\ ({\rm{resp.}}\ J_{\varphi})$ is the modular (resp. modular conjugation) operator of $\varphi$.  $\mathcal{P}_M^{\natural}$ induces a partial order on $H$ by 
\[\xi\le \eta\Leftrightarrow \xi-\eta\in \mathcal{P}_M^{\natural},\]
and for any vector $\xi\in H_{\rm{sa}}:=\{\eta\in H;\ J\eta=\eta\}$, there exist unique vectors $\xi_+,\xi_-\in \mathcal{P}_M^{\natural}$ such that $\xi=\xi_+-\xi_-$ and $\xi_+\perp \xi_-$. Moreover, $H$ is linearly spanned by $\mathcal{P}_M^{\natural}$.   
This order structure was intensively studied in \cite{Connes74,Araki}. Among others, it holds that \cite{Araki} every positive $\varphi \in M_*$ is represented as $\varphi=\omega_{\xi_{\varphi}}=\nai{\ \cdot\ \xi_{\varphi}}{\xi_{\varphi}}$ by a unique vector $\xi_{\varphi}\in \mathcal{P}_M^{\natural}$ and we have the Araki-Powers-St\o rmer inequality: 
\[\|\xi_{\varphi}-\xi_{\psi}\|^2\le \|\varphi-\psi\|\le \|\xi_{\varphi}-\xi_{\psi}\|\|\xi_{\varphi}+\xi_{\psi}\|.\]
The readers are referred to the above papers for more detailed information on the natural cone. 
\subsection{The Ocneanu and the Groh-Raynaud Ultraproduct}
Let $(M_n,\varphi_n)_n$ be a sequence of pairs of $\sigma$-finite von Neumann algebras equipped with normal faithful states. Let $\mathcal{I}_{\omega}:=\mathcal{L}_{\omega}\cap \mathcal{L}_{\omega}^*$, where
\[\mathcal{L}_{\omega}=\mathcal{L}_{\omega}(M_n,\varphi_n):=\{(x_n)_n\in \ell^{\infty}(\mathbb{N},M_n); \|x_n\|_{\varphi_n}\stackrel{n\to \omega}{\to}0\},\]
and let $\mathcal{M}^{\omega}$ be the multiplier algebra of $\mathcal{I}_{\omega}$. The {\it Ocneanu ultraproduct} $(M_n,\varphi_n)^{\omega}$ of $(M_n,\varphi_n)_n$ is the quotient algebra $\mathcal{M}^{\omega}/\mathcal{I}_{\omega}$. For each $n\in \mathbb{N}$, let $H_n=L^2(M_n,\varphi_n)$ be the GNS Hilbert space of $(M_n,\varphi_n)$, and let $H_{\omega}$ be the ultraproduct of $(H_n)_n$. Define a *-representation $\pi_{\omega}\colon \ell^{\infty}(\mathbb{N},M_n)\to \mathbb{B}(H_{\omega})$ by 
\[\pi_{\omega}((x_n)_n)(\xi_n)_{\omega}:=(x_n\xi_n)_{\omega},\ \ \ \ (x_n)_n\in \ell^{\infty}(\mathbb{N},M_n),\ (\xi_n)_{\omega}\in H_{\omega}.\]
The {\it Groh-Raynaud ultraproduct} $\prod^{\omega}M_n$ is the von Neumann algebra $\pi_{\omega}(\ell^{\infty}(\mathbb{N},M_n))''$ acting on $H_{\omega}$. Note that the definition of $\pi_{\omega}$ is slightly different from the one given in \cite[$\S$3]{AndoHaagerup}. 
Let $\xi_n\in H_n$ be the cyclic vector corresponding to $\varphi_n$, and let $\xi_{\omega}=(\xi_n)_{\omega}\in H_{\omega}$. Let $p\in \prod^{\omega}M_n$ be the support projection of the normal state $\nai{\ \cdot\ \xi_{\omega}}{\xi_{\omega}}$ on $\prod^{\omega}M_n$. Then by \cite[Proposition 3.15]{AndoHaagerup},   
$(M_n,\varphi_n)^{\omega}\cong p(\prod^{\omega}M_n)p$. For more details about ultraproducts, see \cite{AndoHaagerup}. 
\subsection{Effros-Mar\'echal Topology on ${\rm{vN}}(H)$}
As explained in the introduction, the {\it Effros-Mar\'echal topology} is the weakest topology on ${\rm{vN}}(H)$ which makes each functional ${\rm{vN}}(H)\ni N\mapsto \|\varphi|_N\|$ continuous $(\varphi \in \mathbb{B}(H)_*)$. As $H$ is separable, this makes ${\rm{vN}}(H)$ a Polish space, and the the topology can be described by the concepts of limsup and liminf of a sequence in ${\rm{vN}}(H)$ (these notions are refinements \cite[Remark 2.11]{HW1} of s-liminf and w-limsup given in \cite{Tsukada}). Since we only consider the case where $H$ is separable, we use the following simplified equivalent definitions (see \cite[$\S$2]{HW1} for the original definition).
\begin{definition}
Let $(M_n)_n\subset {\rm{vN}}(H)$.  
\begin{itemize}
\item[(1)] $\displaystyle \liminf_{n\to \infty}M_n$ is the set of all $x\in \mathbb{B}(H)$ for which there exists $(x_n)_n\in \ell^{\infty}(\mathbb{N},M_n)$ that converges to $x$ *-strongly.
\item[(2)] $\displaystyle \limsup_{n\to \infty}M_n$ is the von Neumann algebra generated by the set of all $x\in \mathbb{B}(H)$ for which there exists $(x_n)_n\in \ell^{\infty}(\mathbb{N},M_n)$ such that $x$ is a weak-limit point of $\{x_n\}_{n=1}^{\infty}$. 
\end{itemize}
\end{definition} 
It is proved in \cite[Theorems 2.8, 3.5]{HW1} that for a sequence $\{M_n\}_{n=1}^{\infty}\subset {\rm{vN}}(H)$ and $M\in {\rm{vN}}(H)$, 
\eqa{
\lim_{n\to \infty}M_n=M\text{\ in\ }{\rm{vN}}(H)&\Leftrightarrow \liminf_{n\to \infty}M_n=M=\limsup_{n\to \infty}M_n,\\
\left (\limsup_{n\to \infty}M_n\right )'&=\liminf_{n\to \infty}M_n'.
}
We will also make use of the ultralimit version of the above concepts.
\begin{definition}
Let $(M_n)_n\subset {\rm{vN}}(H)$. We define 
\begin{itemize}
\item[(1)] $\displaystyle \liminf_{n\to \omega}M_n$ is the set of all $x\in \mathbb{B}(H)$ for which there exists $(x_n)_n\in \ell^{\infty}(\mathbb{N},M_n)$ such that $\displaystyle x=\text{so}^*\text{- }\lim_{n\to \omega}x_n$.
\item[(2)] $\displaystyle \limsup_{n\to \omega}M_n$ is the von Neumann algebra generated by the set of all $x\in \mathbb{B}(H)$ for which there exists $(x_n)_n\in \ell^{\infty}(\mathbb{N},M_n)$ such that $\displaystyle x=\text{wo -}\lim_{n\to \omega}x_n$.  
\end{itemize}
\end{definition}
Again they are different from the original definitions. It also holds that \cite[Lemma 5.2]{HW2} 
\eqa{
M=\lim_{n\to \omega}M_n\text{\ in\ vN}(H)&\Leftrightarrow \liminf_{n\to \omega}M_n=M=\limsup_{n\to \omega}M_n,\\
\left (\limsup_{n\to \omega}M_n\right )'&=\liminf_{n\to \omega}M_n'.
}
Finally, we will consider the following subsets of vN$(H)$ (see the introduction of \cite{HW2} for more details):
\begin{itemize}
\item $\mathcal{F}$, the set of factors acting on $H$.
\item $\mathcal{F}_X$, the set of factors of type $X$ acting on $H$, where $X$ runs among the standard numberings of the types of factors, such as II$_1$ or III$_{\lambda}$.
\item $\mathcal{F}_{\rm{inj}}$, the set of injective factors acting on $H$.
\item $\mathcal{F}^{\rm{st}}$, the set of factors acting standardly on $H$.
\item ${\rm{SA}}(M)$ the set of von Neumann subalgebras of $M\in \text{vN}(H)$.  
\end{itemize}
For a fixed normal faithful state $\varphi$ on $M$, we denote ${\rm{SA}}_{\varphi}(M)$ to be the set of those $N\in {\rm{SA}}(M)$ satisfying $\sigma_t^{\varphi}(N)=N\ (t\in \mathbb{R})$ (see \cite[$\S$2]{HW1}).  
%We also make use of self-explanatory extensions of the above notation, such as $\mathcal{F}_{\rm{inj}}^{\rm{st}}=\mathcal{F}_{{\rm{inj}}}\cap \mathcal{F}^{{\rm{st}}}$. 
\section{Embedding into the Ocneanu Ultraproducts and Effros-Mar\'echal Topology}
In this section, we consider a separable infinite-dimensional Hilbert space $H$. An {\it embedding} is understood to be a unital normal injective *-homomorphism of one von Neumann algebra into another. The main technical part of this section deals with versions of the following statements about von Neumann algebras $M,N\in {\rm{vN}}(H)$:
\begin{itemize}
\item[1.] There is a sequence $(\psi_n)_n\subset S_{\text{nf}}(M)$, an embedding of $i\colon N\to (M,\psi_n)^{\omega}$ and a faithful normal conditional expectation $\varepsilon$ of $(M,\psi_n)^{\omega}$ onto $i(N)$;
\item[2.] There is a separable Hilbert space $K$, a sequence of isomorphic copies $M_n\in \text{vN}(K)$ of $M$ and an isomorphic copy $N_0\in \text{vN}(K)$ of $N$ such that $M_n\to N_0$ in the Effros-Mar\'echal topology;
%\item[3.] $N\subset M$ and there is a (possibly non-normal) faithful conditional expectation of $M$ onto $N$.
\end{itemize}
We  prove that 1. is equivalent to 2. (Theorem \ref{UW4.1.1} and Theorem \ref{UW4.1.3}) As an application we give a new characterization of the closure of the set of injective factors in $\text{vN}(H)$ (see \cite[\S 4-5]{HW2} for more details). 
%Then in Theorem \ref{UW4.2.5}, we prove that 3. implies 1. (hence 2.). 
%As an application, a new proof of the density of type II-factors (cf. \cite[Theorem 3.5]{HW1}) is given.
\subsection{Approximation Theorem and The Closure of Injective Factors.}
We begin by the following
\begin{theorem}\label{UW4.1.1} Let $N\in {\rm{vN}}(H)$. Assume we are given a sequence $(M_n)\subset {\rm{vN}}(H)$ such that $M_n\to N$ in {\rm{vN}}$(H)$. Fix $\chi \in S_{\rm{nf}}(\mathbb{B}(H))$, and let $\psi_n:=\chi|_{M_n}, \varphi:=\chi|_N$. Then there exists an embedding $i$ of $N$ into $(M_n,\psi_n)^{\omega}$ such that $\varphi=(\psi_n)^{\omega}\circ i$, and a normal faithful conditional expectation $\varepsilon$ of $(M_n,\psi_n)^{\omega}$ onto $i(N)$ such that $(\psi_n)^{\omega}\circ \varepsilon =(\psi_n)^{\omega}$. 
\end{theorem}
\begin{proof}
Let $x\in N$. By \cite[Theorem 2.8]{HW1}, $N=\liminf_{n\to \infty}M_n$ and we may choose $(x_n)_n\in \ell^{\infty}(\mathbb{N},M_n)$ satisfying  $x_n\stackrel{\text{so}*}{\longrightarrow }x$. It is clear that $(x_n)_n\in \mathcal{M}^{\omega}(M_n,\psi_n)$ because of the joint continuity of operator product in strong*-topology on bounded sets, and also that $(x_n)^{\omega}$ is independent of the choice of the sequence $(x_n)_n$ converging to $x$. Hence we may define $i\colon N\to (M_n,\psi_n)^{\omega}$ by $i(x)=(x_n)^{\omega}$, which is clearly an injective *-homomorphism, and as
\[(\psi_n)^{\omega}\circ i(x)=\lim_{n\to \omega}\psi_n(x_n)=\lim_{n\to \omega}\chi (x_n)=\varphi(x),\]
we have $(\psi_n)^{\omega}\circ i=\varphi$. In particular, $i$ is normal.
Next, for $y=(x_n)^{\omega}\in (M_n,\psi_n)^{\omega}$, let $x:={\rm{wo}-}\lim_{n\to \omega}x_n$. Then $x\in \limsup_{n\to \infty}M_n=N$ by \cite[Theorem 2.8]{HW1}. We then define $\varepsilon (y):=i(x)$. 
Let $x\in N$ and suppose $i(x)=(x_n)^{\omega}$ for $(x_n)_n\in \ell^{\infty}(\mathbb{N},M_n)$. Then as $x_n\stackrel{{\rm{so*}}}{\to}x$, we have $\varepsilon\circ i(x)=\varepsilon ((x_n)^{\omega})=i(x)$, whence $\varepsilon\circ i=i$. Therefore $\varepsilon$ is a faithful conditional expectation of $(M_n,\psi_n)^{\omega}$ onto $i(N)$
 Next, let $(x_n)^{\omega}\in (M_n,\psi_n)^{\omega}$ and let $x={\rm{wo}-}\lim_{n\to \omega}x_n$. Suppose $i(x)=(x_n')^{\omega}$, where $(x_n')_n\in \ell^{\infty}(\mathbb{N},M_n)$. Then 
\eqa{
(\psi_n)^{\omega}\circ \varepsilon ((x_n)^{\omega})&=(\psi_n)^{\omega}\circ i(x)=(\psi_n)^{\omega}((x_n')^{\omega})\\
&=\lim_{n\to \omega}\psi_n(x_n')=\lim_{n\to \omega}\chi(x_n')=\chi(x)\\
&=\lim_{n\to \omega}\chi(x_n)=(\psi_n)^{\omega}((x_n)^{\omega}),
}
whence $(\psi_n)^{\omega}\circ \varepsilon =(\psi_n)^{\omega}$. Since $(\psi_n)^{\omega}$ is normal and faithful, $\varepsilon$ is also normal. 
\end{proof}
\begin{corollary}\label{UW4.1.2}
Let $M, N\in {\rm{vN}}(H)$, and assume $(M_n)\subset {\rm{vN}}(H)$ has $M_n\cong M (n\in \mathbb{N})$ and $M_n\to N$. Then there exists a sequence $(\psi_n)_n\subset S_{\rm{nf}}(M)$ and an embedding $i$ of $N$ into $(M,\psi_n)^{\omega}$, and a normal faithful conditional expectation of $(M,\psi_n)^{\omega}$ onto $i(N)$.
\end{corollary}
\begin{proof}
Let $\varphi \in S_{\text{nf}}(N)$. Using \cite[Lemma 5.6]{HW1}, we find $\chi \in S_{\text{nf}}(\mathbb{B}(H))$ such that $\varphi=\chi|_N$. Let $\psi_n':=\chi|_{M_n}\in S_{\text{nf}}(M_n)$ for each $n\in \mathbb{N}$. For each $n\in \mathbb{N}$, take a *-isomorphism $\Phi_n\colon M\to M_n$ and let $\psi_n:=\psi_n'\circ \Phi_n$. Then $\Phi((x_n)^{\omega})=(\Phi_n(x_n))^{\omega}$ defines a *-isomorphism $\Phi\colon  (M,\psi_n)^{\omega}\to (M_n,\psi_n')^{\omega}$. By Theorem \ref{UW4.1.1}, there exists an embedding $i'\colon N\to (M_n,\psi_n')^{\omega}$ and a normal faithful conditional expectation $\varepsilon'\colon (M_n,\psi_n')^{\omega}\to i'(N)$. Then $i:=\Phi^{-1}\circ i'\colon N\to (M,\psi_n)^{\omega}$ and $\varepsilon:=\Phi^{-1}\circ \varepsilon'\circ \Phi\colon (M,\psi_n)^{\omega}\to i(N)$ are the required embedding and the normal faithful conditional expectation.  
\end{proof}
We now prove a partial converse to Theorem \ref{UW4.1.1}, which is at the same time a generalization of \cite[Lemma 5.5]{HW2} to the non-tracial case:
\begin{theorem}\label{UW4.1.3}
Let $\{M_n\}_{n=1}^{\infty}\subset {\rm{vN}}^{{\rm{st}}}(H)$ and $N\in {\rm{vN}}^{{\rm{st}}}(H)$. Let $\psi_n\in S_{\rm{nf}}(M_n)\ (n\in \mathbb{N})$. Assume that we are given an embedding $i\colon N\to (M_n,\psi_n)^{\omega}$ and a normal faithful conditional expectation $\varepsilon\colon (M_n,\psi_n)^{\omega}\to i(N)$. Then there exists $(u_n)_n\subset \mathcal{U}(H)$ and a strictly increasing sequence $\{n_k\}_{k=1}^{\infty}\subset \mathbb{N}$ such that $u_{n_k}M_{n_k}u_{n_k}^*\to N$ in ${\rm{vN}}(H)$.  
\end{theorem}

\begin{proof}
For simplicity, we use notations $\widetilde{M}=(M_n,\psi_n)^{\omega}$ and $\tilde{\psi}=(\psi_n)^{\omega}$. 
Let $\tilde{\varphi}:=\tilde{\psi}\circ \varepsilon\in S_{\text{nf}}(\tilde{M})$. Then by \cite[Corollary 3.29]{AndoHaagerup}, there exists $\varphi_n\in S_{\rm{nf}}(M_n)\ (n\in \mathbb{N})$ such that $(M_n,\varphi_n)^{\omega}=\widetilde{M}$ and $\tilde{\varphi}=(\varphi_n)^{\omega}$. Put $\varphi:=\tilde{\varphi}|_{i(N)}$. Since each $M_n$ acts standardly on $H$, there exists a cyclic and separating vector $\xi_n\in H$ for $M_n$ such that $\varphi_n(x)=\nai{x\xi_n}{\xi_n}\ (x\in M_n)$ for $n\in \mathbb{N}$. Let $\tilde{\xi}:=(\xi_n)_{\omega}\in H_{\omega}$. Define as in \cite[Theorem 3.7]{AndoHaagerup} an isometry $w\colon L^2(\widetilde{M},\tilde{\psi})\to H_{\omega}$ given as the unique extension of $(x_n)^{\omega}\xi_{\tilde{\varphi}}\mapsto (x_n\xi_n)_{\omega}\ ((x_n)^{\omega}\in \widetilde{M})$. 
Put $K:=w(L^2(\widetilde{M},\tilde{\varphi}))$. Let $J_n\colon H\to H$ be the modular conjugation associated with the standard representation of $M_n$ on $H$. Let $J_{\omega}:=(J_n)_{\omega}\colon H_{\omega}\to H_{\omega}$. Then by \cite[Theorems 3.7 and 3.19]{AndoHaagerup}, $J_{\omega}$ is the modular conjugation associated with the standard representation of $\prod^{\omega}M_n$ on $H_{\omega}$, and $J_{\omega}|_K=wJ_{\widetilde{M}}w^*$. Here, $J_{\widetilde{M}}$ is the modular conjugation associated with the standard representation of $\widetilde{M}$ on $L^2(\widetilde{M},\tilde{\varphi})$. Since $\varepsilon (\widetilde{M})=i(N)$ and $\varphi=\tilde{\varphi}|_{i(N)}$, $\tilde{\varphi}=\varphi \circ \varepsilon$ holds. Thus we may regard $L^2(i(N),\varphi)$ as a closed subspace of $L^2(\widetilde{M},\tilde{\varphi})$ (thus $\xi_{\varphi}=\xi_{\tilde{\varphi}}$). In this case by the argument in \cite{Takesaki2}  $J_{i(N)}=J_{\widetilde{M}}|_{L^2(i(N),\varphi)}$ holds, where $J_{i(N)}$ is the modular conjugation associated with the natural representatin $\pi_{\varphi}$ of $N$ on $L^2(i(N),\varphi)$ given by $\pi_{\varphi}(x)i(y)\xi_{\varphi}=i(xy)\xi_{\varphi}\ (x,y\in N)$. Put $L:=w(L^2(i(N),\varphi))\subset K\subset H_{\omega}$ and $w_0:=w|_{L^2(i(N),\varphi)}$. Let $P_L$ (resp. $P_K$) be the projection of $H_{\omega}$ onto $L$ (resp. $K$), and let $e_L$ be the projection of $K$ onto $L$. Let $e$ be the projection of $L^2(\widetilde{M},\tilde{\varphi})$ onto $L^2(i(N),\varphi)$.   

\[
\xymatrix{
L^2(\widetilde{M},\tilde{\varphi})\ar[d]^e_{\cup }\ar[r]^{\ \ w} & K\ar[d]^{e_L}_{\cup} & \ar[l]_{P_K}^{\subset}\ar[dl]^{P_L}H_{\omega}\\
L^2(i(N),\varphi)\ar[r]^{\ \ \ \ w_0} & L & 
}
%\xymatrix{
%L^2(i(N),\varphi)\ar[d]_{w_0}\ar@{^{(}->}[r]_{}\ar[r] 
%& L^2(\widetilde{M},\tilde{\varphi})\ar[d]_{w}& \\
%L\ \ \ar@{^{(}->}[r]^{} & K\ \  \ar@{^{(}->}[r]& H_{\omega}
%}
\]
Since $N_*$ is separable, $L$ is a separable subspace of $H_{\omega}$. 
So by \cite[Lemma 5.1]{HW2}, there exist unitaries $v_n\in \mathcal{U}(L,H)\ (n\in \mathbb{N})$ such that $\xi=(v_n\xi)_{\omega}$ holds for every $\xi \in L$. 
Recall that to construct the Groh-Raynaud ultraproduct $\prod^{\omega}M_n$, we used $\pi_{\omega}\colon \ell^{\infty}(\mathbb{N},M_n)\to \mathbb{B}(H_{\omega})$ given by $\pi_{\omega}((x_n)_n)(\xi_n)_{\omega}=(x_n\xi_n)_{\omega}$ for $(x_n)_n\in \ell^{\infty}(\mathbb{N},M_n)$ and $(\xi_n)_{\omega}\in H_{\omega}$. 
Define an analogous map $\hat{\pi}_{\omega}\colon \ell^{\infty}(\mathbb{N},M_n')\to \mathbb{B}(H_{\omega})$. 
\\ \\
\textbf{Claim 1.} Let $(x_n)_n\in \ell^{\infty}(\mathbb{N},M_n)$ and $(x_n')_n\in \ell^{\infty}(\mathbb{N},M_n')$. Then as elements of $\mathbb{B}(L)$, the following holds. 
\begin{align}
\text{wo -}\lim_{n\to \omega}v_n^*x_nv_n=P_L\pi_{\omega}((x_n)_n)P_L,\label{eq: P_LB(H)P_L part 1}\\
\text{wo -}\lim_{n\to \omega}v_n^*x_n'v_n=P_L\hat{\pi}_{\omega}((x_n')_n)P_L.\label{eq: P_LB(H)P_L part 2}
\end{align}

%\textbf{Claim.}\ $\displaystyle \lim_{n\to \omega}v_n^*M_nv_n=\pi_L(N)$.\\

Let $\displaystyle y:=\text{wo -}\lim_{n\to \omega}v_n^*x_nv_n$. Clearly $v_n^*x_nv_n\in P_L\mathbb{B}(H_{\omega})P_L$, and for $\xi,\eta\in L$, we have 
\eqa{
\nai{y\xi}{\eta}&=\lim_{n\to \omega}\nai{x_nv_n\xi}{v_n\eta}\\
&=\nai{\pi_{\omega}((x_n)_n)(v_n\xi)_{\omega}}{(v_n\eta)_{\omega}}\\
&=\nai{\pi_{\omega}((x_n)_n)\xi}{\eta},
}
whence (\ref{eq: P_LB(H)P_L part 1}) holds. (\ref{eq: P_LB(H)P_L part 2}) can be proved similarly.  
\\ \\
\textbf{Claim 2.} The following holds (as elements in $\mathbb{B}(L)$):
\begin{align}
P_L\pi_{\omega}(\ell^{\infty}(\mathbb{N},M_n))P_L&\subset \pi_L(N),\label{eq: P_LB(H)P_L is in pi_L}\\
P_L\hat{\pi}_{\omega}(\ell^{\infty}(\mathbb{N},M_n'))P_L&\subset \pi_L(N)',\label{eq: P_LB(H)'P_L is in pi_L'}
\end{align}
where $\pi_L$ is the natural standard action of $N$ on $L$ unitarily equivalent to $\pi_{\varphi}$, i.e., $\pi_L(x)=w_0\pi_{\varphi}(x)w_0^*\ (x\in N)$.
By \cite[Theorem 3.7, Corollary 3.28]{AndoHaagerup}, $P_K=ww^*, w^*(\prod^{\omega}M_n)w=\widetilde{M}$ and $\rho\colon \widetilde{M}\ni \pi_{\tilde{\varphi}}(x)\mapsto wxw^*\in P_K(\prod^{\omega}M_n)P_K$ is a *-isomorphism. Note that by the existence of $\varepsilon\colon \widetilde{M}\to i(N)$, $e$ satisfies $e\pi_{\tilde{\varphi}}((x_n)^{\omega})e=\pi_{\varphi}(\varepsilon((x_n)^{\omega}))e=e\pi_{\varphi}(\varepsilon((x_n)^{\omega}))$ for $(x_n)^{\omega}\in \widetilde{M}$. Now let $(x_n)_n\in \ell^{\infty}(\mathbb{N},M_n)$. Then by \cite[Corollary 3.16]{AndoHaagerup}, there exists $(a_n)_n\in \mathcal{M}^{\omega}(M_n,\varphi_n)$, $(b_n)_n\in \mathcal{L}_{\omega}(M_n,\varphi_n)$ and $(c_n)_n\in \mathcal{L}_{\omega}(M_n,\varphi_n)^*$ such that $x_n=a_n+b_n+c_n\ (n\in \mathbb{N})$, and moreover by (the proof of) \cite[Theorem 3.7, Corollary 3.28]{AndoHaagerup}, the following equalities hold:
\[P_K\pi_{\omega}((x_n)_n)P_K=P_K(\pi_{\omega}((a_n)_n)P_K=w\pi_{\tilde{\varphi}}((a_n)^{\omega})w^*.\]
Then as $P_L=e_LP_K$, it holds that for $\xi \in L$, 
\eqa{
P_L\pi_{\omega}((x_n)_n)P_L\xi &=e_LP_K\pi_{\omega}((x_n)_n)P_Ke_LP_K\xi \\
&=e_Lw\pi_{\tilde{\varphi}}((a_n)^{\omega})w^*e_L\xi\\
&=w_0e\pi_{\tilde{\varphi}}((a_n)^{\omega})ew_0^*\xi\\
&=w_0\pi_{\varphi}(\varepsilon ((a_n)^{\omega}))w_0^*\xi\\
&=\pi_L(i^{-1}\circ \varepsilon ((a_n)^{\omega}))\xi,
}
so $P_L\pi_{\omega}((x_n)_n)P_L=\pi_L(i^{-1}\circ \varepsilon ((a_n)^{\omega}))\in \pi_L(N)$ and (\ref{eq: P_LB(H)P_L is in pi_L}) is proved. For (\ref{eq: P_LB(H)'P_L is in pi_L'}), let $(x_n')_n\in \ell^{\infty}(\mathbb{N},M_n)$. Then $x_n'=J_nx_nJ_n$, $x_n:=J_nx_n'J_n\in M_n\ (n\in \mathbb{N})$, and $\hat{\pi}_{\omega}((x_n')_n)=J_{\omega}\pi_{\omega}((x_n)_n)J_{\omega}$. Since $P_K=pJ_{\omega}pJ_{\omega}$ \cite[Corollary 3.28 (i)]{AndoHaagerup}, where $p=\text{supp}((\varphi_n)_{\omega})\in \prod^{\omega}M_n$, $P_KJ_{\omega}=J_{\omega}P_K$ holds. Choose $(a_n)_n\in \mathcal{M}^{\omega}(M_n,\varphi_n)$ for $(x_n)_n$ as above so that $w\pi_{\tilde{\varphi}}((a_n)^{\omega})w^*=P_K\pi_{\omega}((x_n)_n)P_K$ holds. Thus, it holds that (as elements in $\mathbb{B}(L)$)
\eqa{
P_L\hat{\pi}_{\omega}((x_n'))P_L&=e_LP_KJ_{\omega}\pi_{\omega}((x_n)_n)J_{\omega}P_Ke_L\\
&=e_LJ_{\omega}P_K\pi_{\omega}((x_n)_n)P_KJ_{\omega}e_L\\
&=J_Le_Lw\pi_{\tilde{\varphi}}((a_n)^{\omega})w^*e_LJ_L\\
&=J_L\pi_L(i^{-1}\circ \varepsilon ((a_n)^{\omega}))J_L\in \pi_L(N)',
}
and (\ref{eq: P_LB(H)'P_L is in pi_L'}) is proved.
 \\ \\
\textbf{Claim 3.} $\displaystyle \lim_{n\to \omega}v_n^*M_nv_n=\pi_L(N)$.\\
To show the claim, by \cite[Lemma 5.2]{HW2}, it suffices to show the following:
\begin{align}
\limsup_{n\to \omega}v_n^*M_nv_n&\subset \pi_L(N),\label{eq: limsup v_n^*M_nv_n}\\
\limsup_{n\to \omega}v_n^*M_n'v_n&\subset \pi_L(N)'.\label{eq: limsup v_n^*M_n'v_n} 
\end{align}
To prove (\ref{eq: limsup v_n^*M_nv_n}) and (\ref{eq: limsup v_n^*M_n'v_n}), it is enough to show that for every $(x_n)_n\in \ell^{\infty}(\mathbb{N},M_n)$ and $(x_n')_n\in \ell^{\infty}(\mathbb{N},M_n')$,  
\begin{align}
\text{wo -}\lim_{n\to \omega}v_n^*x_nv_n&\in \pi_L(N),\label{eq: v_n^*x_nv_n is in pi_L(N)}\\
\text{wo -}\lim_{n\to \omega}v_n^*x_n'v_n&\in \pi_L(N)'\label{eq: v_n^*x_n'v_n is in pi_L(N)'}.
\end{align}
But they are the consequences of Claim 1 and Claim 2.

Finally, as the representation $\pi_L$ is standard and $N\in {\rm{vN}}^{\rm{st}}(H)$, there exists $u\in \mathcal{U}(L,H)$ such that $u\pi_L(x)u^*=x\ (x\in N)$. Then by Claim 3, we have $u_nMu_n^*\to N\ (n\to \omega)$ in vN$(H)$, where $u_n=uv_n^*\in \mathcal{U}(H)$. From this and the fact that ${\rm{vN}}(H)$ is a separable metrizable space, we may find a subsequence $\{n_k\}_{k=1}^{\infty}$ so that $u_{n_k}Mu_{n_k}^*\to N\ (k\to \infty)$ in vN$(H)$.  
\end{proof}

In \cite[Theorem 5.8]{HW2}, the second and the third-named authors proved that a type II$_1$ factor is in $\overline{\mathcal{F}_{\text{inj}}}$ (the closure of injective factors) if and only if it embeds into $R^{\omega}$, where $R$ is the injective II$_1$ factor. Let $R_{\infty}$ (resp. $R_{\lambda}$) denote the injective factor of type III$_1$ (resp. type III$_{\lambda}$\ $(0<\lambda<1)$). As an application of the theorems above, we get the following result for general von Neumann algebras:
\begin{theorem}\label{UW4.1.4} Let $N\in {\rm{vN}}(H)$ and let $0<\lambda<1$. Then the following conditions are equivalent:
\begin{list}{}{}
\item[{\rm{(i)}}] $N\in \overline{\mathcal{F}_{\rm{inj}}}$.
\item[{\rm{(ii)}}] There is an embedding $i\colon N\to R_{\infty}^{\omega}$ and a normal faithful conditional expectation $\varepsilon\colon R_{\infty}^{\omega}\to i(N)$.
\item[{\rm{(iii)}}] There is an embedding $i\colon N\to R_{\lambda}^{\omega}$ and a normal faithful conditional expectation $\varepsilon\colon R_{\lambda}^{\omega}\to i(N)$.
\item[{\rm{(iv)}}] There is $\{k_n\}_{n=1}^{\infty}\subset \mathbb{N}$, a sequence $\varphi_n\in S_{\rm{nf}}(M_{k_n}(\mathbb{C}))$ such that $N$ admits an embedding $i\colon N\to (M_{k_n}(\mathbb{C}),\varphi_n)^{\omega}$ and a normal faithful conditional expectation $\varepsilon\colon (M_{k_n}(\mathbb{C}),\varphi_n)^{\omega}\to i(N)$. 
\end{list}
\end{theorem}
\begin{proof}
(i)$\Rightarrow$(ii): By \cite[Theorem 2.10 (ii)]{HW2}, $\mathcal{F}_{{\rm{III}}_1}\cap \mathcal{F}_{\rm{inj}}$ is dense in $\mathcal{F}_{\rm{inj}}$. Therefore condition (i) implies that there exists a sequence $(M_n)_n$ of injective type III$_1$ factors on $H$ such that $M_n\to N$. Then as $M_n\cong R_{\infty}\ (n\in \mathbb{N})$, by Corollary \ref{UW4.1.2}, there exists a sequence $(\psi_n)_n\subset S_{\rm{nf}}(R_{\infty})$ such that there exists an embedding $i\colon N\to (R_{\infty},\psi_n)^{\omega}$ and a normal faithful conditional expectation $\varepsilon\colon (R_{\infty},\psi_n)^{\omega}\to i(N)$. By \cite[Theorem 6.11]{AndoHaagerup}, we have $(R_{\infty},\psi_n)^{\omega}\cong R_{\infty}^{\omega}$. Therefore (ii) holds. 
%$N$ is the limit of a sequence of injective III$_1$ factors, i.e., of isomorphic copies of $R_{\infty}$. Then (ii) follows by Corollaries \ref{UW4.1.2} and \ref{UW3.1.3}.
(ii)$\Rightarrow$(i): Let $K_1,K_2$ be separable infinite-dimensional Hilbert spaces such that $\widetilde{N}:=N\overline{\otimes}\mathbb{B}(K_1)\overline{\otimes}\mathbb{C}1_{K_2}$ acts standardly on $H\otimes K$, $K:=K_1\otimes K_2$. Then we obtain an embedding $i'=i\otimes \text{id}\otimes \text{id}\colon \widetilde{N}\to Q^{\omega}$, $Q:=R_{\infty}\overline{\otimes}\mathbb{B}(K_1)\overline{\otimes}\mathbb{C}1_{K_2}$ and a normal faithful conditional expectation $\varepsilon'=\varepsilon\otimes \text{id}\otimes \text{id}\colon Q^{\omega}\to i'(\widetilde{N})$. Here we used the fact that $Q^{\omega}\cong R_{\infty}^{\omega}\overline{\otimes}\mathbb{B}(K_1)\overline{\otimes}\mathbb{C}1_{K_2}$ (cf. \cite[Lemma 2.8]{MaTo}). Since $Q,\widetilde{N}\in \text{vN}^{\text{st}}(H\otimes K)$ (note that $Q$ is a type III factor), by Theorem \ref{UW4.1.3}, there exist $(w_n)_n\subset \mathcal{U}(H\otimes K)$ such that $w_n^*Qw_n\to \widetilde{N}$ in vN$(H\otimes K)$. In particular, as $Q\cong R_{\infty}$, $\widetilde{N}\in \overline{\mathcal{F}_{\rm{inj}}}$ holds.
Choose $v_0\in \mathcal{U}(H\otimes K,H)$. Then by \cite[Lemma 2.4]{HW2}, there exist $(u_n)_n\subset \mathcal{U}(H\otimes K)$ such that $v_0u_n^*\widetilde{N}u_nv_0^*\to N$ in vN$(H)$. As $v_0u_n^*\widetilde{N}u_nv_0^*\in \overline{\mathcal{F}_{\rm{inj}}}$, this shows that $N\in \overline{\mathcal{F}_{\rm{inj}}}$ holds.\\
(i)$\Leftrightarrow$(iii) holds similarly, again using \cite[Corollary 2.11]{HW2} that $\mathcal{F}_{{\rm{III}}_\lambda}\cap \mathcal{F}_{\rm{inj}}$ is dense in $\mathcal{F}_{\rm{inj}}$.\\
(i)$\Rightarrow$(iv): Assume $N\in \overline{\mathcal{F}_{\rm{inj}}}$. Again by (the proof of) \cite[Corollary 2.11]{HW2}, the set $\mathcal{F}_{{\rm{I}}_{{\rm{fin}}}}$ of finite type I factors is dense in $\mathcal{F}_{\rm{inj}}$. 
Therefore there is $\{k_n\}_{n=1}^{\infty}\subset \mathbb{N}$ and $N_n\in {\rm{vN}}(H)$ with $N_n\cong M_{k_n}(\mathbb{C})\ (n\in \mathbb{N})$ such that $N_n\to N$ in vN$(H)$. Then by Theorem \ref{UW4.1.1}, there is $\varphi_n\in S_{{\rm{nf}}}(M_{k_n}(\mathbb{C}))\ (n\in \mathbb{N})$, an embedding $i\colon N\to (M_{k_n}(\mathbb{C}),\varphi_n)^{\omega}$, and a normal faithful conditional expectation $\varepsilon\colon (M_{k_n}(\mathbb{C}),\varphi_n)^{\omega}\to i(N)$. (iv)$\Rightarrow$(ii): Assume (iv).  Choose $\psi\in S_{\rm{nf}}(R_{\infty})$. Then for each $n\in \mathbb{N}$, we haven an embedding $i_n\colon M_{k_n}(\mathbb{C})\to R_{\infty}\otimes M_{k_n}(\mathbb{C})$ given by $x\mapsto 1\otimes x$ and a normal faithful conditional expectation $\varepsilon_n\colon R_{\infty}\otimes M_{k_n}(\mathbb{C})\to M_{k_n}(\mathbb{C})$ given by the left slice-map of $\psi$ ($x\otimes y\mapsto \psi(x)y$). Since $R_{\infty}\otimes M_{k_n}(\mathbb{C})\cong R_{\infty}$ we may regard $i_n\colon M_{k_n}(\mathbb{C})\to R_{\infty}$ and $\varepsilon_n\colon R_{\infty}\to M_{k_n}(\mathbb{C})$. Let $\tilde{\varphi}_n:=\varphi_n\circ \varepsilon_n$. 
%By \cite{HermanTakesaki}, there exists $\psi \in S_{\rm{nf}}(R_{\infty})$ with $(R_{\infty})_{\psi}=\mathbb{C}$. Put $\tilde{\psi}_n:=\psi\otimes \text{tr}_{k_n}$, where $\text{tr}_{k_n}$ is the tracial state on $M_{k_n}(\mathbb{C})$. Then it is easy to see that $(R_{\infty}\otimes M_{k_n}(\mathbb{C}))_{\tilde{\psi}_n}=\mathbb{C}\otimes M_{k_n}(\mathbb{C})$. Therefore there exists an embedding $i_n\colon M_{k_n}(\mathbb{C})\to R_{\infty}$ and a normal faithful conditional expectation $\varepsilon_n\colon R_{\infty}\to i_n(M_{k_n}(\mathbb{C}))$. Put $\tilde{\varphi}_n:=\varphi_n\circ \varepsilon_n\in S_{\rm{nf}}(R_{\infty})$. 
Then by taking the ultraproducts, we have an embedding $i^{\omega}:=(i_n)^{\omega}\colon (M_{k_n}(\mathbb{C}),\varphi_n)^{\omega}\to (R_{\infty},\tilde{\varphi}_n)^{\omega}\cong R_{\infty}^{\omega}$ (cf. \cite[Theorem 6.11]{AndoHaagerup}) and a normal faithful conditional expectation $\varepsilon^{\omega}\colon R_{\infty}^{\omega}\to i^{\omega}((M_{k_n}(\mathbb{C}),\varphi_n)^{\omega})$. Therefore $i^{\omega}\circ i\colon N\to R_{\infty}^{\omega}$ and $\varepsilon\circ \varepsilon^{\omega}\colon R_{\infty}^{\omega}\to i^{\omega}\circ i(N)$ works.
%, or instead (ii)$\Leftrightarrow$(iii$_{\lambda}$) may be proved using the fact that there exists an embedding $i:R_{\infty}\to R_{\lambda}$ (resp. $i:R_{\lambda}\to R_{\infty}$) and a normal faithful conditional expectation $\varepsilon:R_{\lambda}\to i(R_{\infty})$ (resp. $\varepsilon: R_{\infty}\to i(R_{\lambda})$). 
\end{proof}
\begin{remark}
Nou \cite{Nou} has shown that $q$-deformed Araki-Woods algebras (in the sense of Hiai \cite{Hiai}, see also \cite{Shlyakhtenko}) have QWEP, whence by the above theorem they are embeddable into $R_{\infty}^{\omega}$ within the range of normal faithful conditional expectations. 
\end{remark}
Fraha-Hart-Sherman \cite{FHS} proved, using model theory, that there exists a II$_1$ factor $M$ with separable predual such that every $N\in \mathcal{F}_{{\rm{II}}_1}$ is embeddable into $M^{\omega}$. 
We show the analogous result for a type III$_1$ factor using the Effros-Mar\'echal topology instead.
\begin{corollary}\label{cor: type III poorman's solution}
There exists a type {\rm{III}}$_1$ factor $M$ with separable predual such that for every $N\in {\rm{vN}}(H)$, there is an embedding $i\colon N\to M^{\omega}$ and a normal faithful conditional expectation $\varepsilon\colon M^{\omega}\to i(N)$.
\end{corollary}
\begin{proof}
Let $\{M_n\}_{n=1}^{\infty}$ be a sequence of type III$_1$ factors on $H$ which is dense in ${\rm{vN}}(H)$(see \cite[Theorem 2.10 (i)]{HW2}). Choose $\varphi_n\in S_{{\rm{nf}}}(M_n)$ for each $n\in \mathbb{N}$. Define $M:=\bigotimes_{n\in \mathbb{N}}(M_n,\varphi_n)$. Then for every $N\in {\rm{vN}}(H)$, there exists a sequence $\{n_k\}_{k=1}^{\infty}$ such that $M_{n_k}\to N\ (k\to \infty)$. Then by Theorem \ref{UW4.1.1}, there is $\psi_n\in S_{\rm{nf}}(M_n)$, an embedding $i'\colon N\to (M_{n_k},\psi_{n_k})^{\omega}$ and a normal faithful conditional expectation $\varepsilon'\colon (M_{n_k},\psi_{n_k})^{\omega}\to i'(N)$.  
But since each $M_n$ is a type III$_1$ factor, $(M_{n_k},\psi_{n_k})^{\omega}\cong (M_{n_k},\varphi_{n_k})^{\omega}$ holds by Connes-St\o rmer transitivity \cite{CS}. There exist an embedding $j_k\colon M_{n_k}\to M$ (into $n_k$-th tensor component of $M$) such that $\varphi\circ j_k=\varphi_{n_k}$, where $\varphi=\bigotimes_{n\in \mathbb{N}}\varphi_n$ and a normal faithful conditional expectation $\varepsilon_k\colon M\to M_{n_k}$. Therefore there is an embedding $j\colon (M_{n_k},\varphi_{n_k})^{\omega}\to (M,\varphi)^{\omega}\cong M^{\omega}$ and a normal faithful conditional expectation $\varepsilon\colon M^{\omega}\to j((M_{n_k},\varphi_{n_k})^{\omega})$ given by the ultraproduct of $\{j_k\}_{k=1}^{\infty}$ and $\{\varepsilon_k\}_{k=1}^{\infty}$. Then $i=j\circ i'\colon N\to M^{\omega}$ and $\varepsilon'\circ \varepsilon\colon M^{\omega}\to i(N)$ works.
\end{proof}
%\begin{corollary}[Farah-Hart-Sherman \cite{FHS}]\label{cor: poormans's CEP}\ 
%There exists a type {\rm{II}}$_1$ factor $M$ with separable predual such that every type {\rm{II}}$_1$ factor $N$ with separable predual admits an embedding $i\colon N\to M^{\omega}$. 
%\end{corollary}
%\begin{proof}
%Let $\{M_n\}_{n=1}^{\infty}$ be a dense subset of $\mathcal{F}_{{\rm{II}}_1}$. Define $M:=\bigotimes_{n\in \mathbb{N}}(M_n,\tau_n)$. Then for every II$_1$ factor $N\in {\rm{vN}}(H)$, there exists a sequence $\{n_k\}_{k=1}^{\infty}$ such that $M_{n_k}\to N\ (k\to \infty)$. Then as in the proof of Theorem \ref{UW4.1.1}, there is an embedding $i\colon N\to (M_{n_k},\tau_{n_k})^{\omega}$ given by $i(x)=(x_k)^{\omega}$, where $(x_k)_k\in \ell^{\infty}(\mathbb{N},M_{n_k})$ is such that $x_k\stackrel{\rm{so}^*}{\to}x$. Since each $M_{n_k}$ embeds into the II$_1$ factor $M$ in a trace-preserving way, $(M_{n_k},\tau_{n_k})^{\omega}$ embeds into the II$_1$ factor $M^{\omega}$. Therefore $N$ embeds into $M^{\omega}$.    
%\end{proof}

Raynaud showed \cite[Proposition 1.14]{Raynaud} that $\prod^{\mathcal{U}}\mathbb{B}(H)$ is not semifinite where $H$ is infinite-dimensional and $\mathcal{U}$ is a free ultrafilter on the set $I$ of all pairs $(E,\varepsilon)$ where $E$ is a finite-dimensional subspace of $\mathbb{B}(H)_*$ and $\varepsilon>0$ partially ordered by $(E_1,\varepsilon_1)\le (E_2,\varepsilon_2)\Leftrightarrow \varepsilon_1\ge \varepsilon_2$ and $E_1\subset E_2$. The proof is based on \cite[Lemma 1.12, Lemma 1.13]{Raynaud} that $\mathbb{B}(H)^{**}$ is not semifinite, and (by local reflexivity) there is an embedding  $i\colon \mathbb{B}(H)^{**}\to \prod^{\mathcal{U}}\mathbb{B}(H)$ and a normal faithful conditional expectation $\varepsilon\colon \prod^{\mathcal{U}}\mathbb{B}(H)\to i(\mathbb{B}(H)^{**})$. Therefore the large index set was required in his argument. 
We show that the same conclusion holds when $I$ is replaced by $\mathbb{N}$ using Effros-Mar\'echal topology.
\begin{corollary}\label{cor: Raynaud for index N}
$\prod^{\omega}\mathbb{B}(H)$ is not semifinite.
\end{corollary}
\begin{proof} Let $M\in \text{vN}(H)$ be an injective type III factor. Then by the proof of \cite[Corollary 2.11]{HW2}, it is in the closure of the set $\mathcal{F}_{{\rm{I}}_{\infty}}$ of type I$_{\infty}$ factors on $H$. Then by Theorem \ref{UW4.1.1}, there is a sequence $\{\varphi_n\}_{n=1}^{\infty}\subset S_{\rm{nf}}(\mathbb{B}(H))$, an embedding $i\colon M\to (\mathbb{B}(H),\varphi_n)^{\omega}$, and a normal faithful conditional expectation $\varepsilon\colon (\mathbb{B}(H),\varphi_n)^{\omega}\to i(M)$. In particular, $(\mathbb{B}(H),\varphi_n)^{\omega}$ is not semifinite by \cite[Theorem 3]{Tomiyama}. Then by \cite[Proposition 3.15]{AndoHaagerup}, $\prod^{\omega}\mathbb{B}(H)$ has a non-semifinite corner $p(\prod^{\omega}\mathbb{B}(H))p\cong (\mathbb{B}(H),\varphi_n)^{\omega}$, $p={\rm{supp}}((\varphi_n)_{\omega})$.
\end{proof}    
\section{Characterizations of QWEP von Neumann Algebras}
In this last section, we establish two characterizations of von Neumann algebras with QWEP. Recall first the definition of QWEP.
\begin{definition}\label{def: WEP, QWEP}
Let $A$ be a C$^*$-algebra.
\begin{list}{}{}
\item[{\rm{(1)}}]
$A$ is said to have the {\it weak expectation property} (WEP for short), if for any faithful representation $A\subset \mathbb{B}(H)$, there exists a unital completely positive (u.c.p. for short) map $\Phi\colon \mathbb{B}(H)\to A^{**}$ such that $\Phi(a)=a$ for all $a\in A$. 
\item[{\rm{(2)}}]
$A$ is said to have the {\it quotient weak expectation property} (QWEP for short), if there exists a surjective *-homomorphism from a C$^*$-algebra $B$ with WEP onto $A$.
\end{list}
\end{definition}
Next theorem characterizes QWEP von Neumann algebras in terms of Effros-Mar\'echal topology.
\begin{theorem}\label{thm: QWEP iff closure of injective factors}
Let $H$ be a separable Hilbert space, and let $M\in {\rm{vN}}(H)$. The following conditions are equivalent. 
\begin{list}{}{}
\item[{\rm{(1)}}] $M\in \overline{\mathcal{F}_{\rm{inj}}}$.
\item[{\rm{(2)}}] $M$ has {\rm{QWEP}}.
\end{list}
\end{theorem}
%\begin{lemma}\label{lem: multiplier algebra has QWEP}
%Let $A$ be a {\rm{C}}$^*$-algebra with {\rm{QWEP}}, $B$ be a hereditary {\rm{C}}$^*$-subalgebra of $A$, $M$ be the multiplier algebra of $B$. Then $M$ has {\rm{QWEP}}.
%\end{lemma}
Although it is not strightforward to prove that the Ocneanu ultraproduct of QWEP von Neumann algebras has QWEP by the definition, it is easy to do so by passing to the Groh-Raynaud ultraproduct. 
\begin{lemma}\label{lem: Ocneanu UP has QWEP}
Let $\{(M_n,\varphi_n)\}_{n=1}^{\infty}$ be a sequence of $\sigma$-finite von Neumann algebras with faithful normal states. Assume that $M_n$ has {\rm{QWEP}} for each $n\in \mathbb{N}$. Then $(M_n,\varphi_n)^{\omega}$ has {\rm{QWEP}}.
\end{lemma}
\begin{proof}
Since $(M_n,\varphi_n)^{\omega}\cong p(\prod^{\omega}M_n)p, p=\text{supp}((\varphi_n)_{\omega})$, it suffices to show that $\prod^{\omega}M_n$ has QWEP. Since each $M_n$ has QWEP, by \cite[Corollary 3.3 (i)]{Kirchberg}, $\ell^{\infty}(\mathbb{N},M_n)$ has QWEP, so does its quotient $(M_n)_{\omega}$. Therefore $((M_n)_{\omega})^{**}$ has QWEP too \cite[Corollary 3.3 (v)]{Kirchberg}. Since $\prod^{\omega}M_n=\pi_{\omega}(\ell^{\infty}(\mathbb{N},M_n))''$ is a corner of $((M_n)_{\omega})^{**}$, it has QWEP. 
\end{proof}

\begin{proof}[Proof of Theorem \ref{thm: QWEP iff closure of injective factors}]
(1)$\Rightarrow$(2): Assume that $M\in \overline{\mathcal{F_{\rm{inj}}}}$. By Theorem \ref{UW4.1.4}, there is an embedding $i\colon M\to R_{\infty}^{\omega}$ with a normal faithful conditional expectation $\varepsilon\colon R_{\infty}^{\omega}\to i(M)$. By Lemma \ref{lem: Ocneanu UP has QWEP}, $R_{\infty}^{\omega}$ has QWEP, whence $M$ has QWEP. 
(2)$\Rightarrow $(1): Let $M\in {\rm{vN}}(H)$ with QWEP. We reduce the problem to the case where von Neumann algebras involved are of finite type using the reduction method from \cite{HJX}. Consider $G=\mathbb{Z}[\frac{1}{2}]=\{\frac{m}{2^n}; m\in \mathbb{Z},n\in \mathbb{N}\}$ as a countable discrete subgroup of $\mathbb{R}$, and fix $\varphi\in S_{\rm{nf}}(M)$. Conside the dual state $\widehat{\varphi}$ on $N\colon=M\rtimes_{\sigma^{\varphi}}G$. There is a canonical embedding $\pi\colon M\to N$, and as $G$ is countable discrete, there is a normal faithful conditional expectation $\varepsilon\colon N\to \pi(M)$. By \cite[Theorem 2.1]{HJX}, there exists an increasing sequence $\{N_n\}_{n=1}^{\infty}$ of finite von Neumann subalgebras of $N$ with $\sigma$so*-dense union, together with a unique normal faithful conditional expectation $\varepsilon_n\colon N\to N_n\ (n\in \mathbb{N})$ such that 
\[\widehat{\varphi}\circ \varepsilon_n=\widehat{\varphi},\ \ \ \sigma_t^{\widehat{\varphi}}\circ \varepsilon_n=\varepsilon_n\circ \sigma_t^{\widehat{\varphi}},\ \ \ (t\in \mathbb{R}).\]
Moreover, it holds that (by \cite[Lemma 2.7]{HJX})
\begin{equation}
\sigma \text{so*-}\lim_{n\to \infty}\varepsilon_n(x)=x,\ \ \ \ x\in N.\label{eq: expectation converges to id}
\end{equation}
Let $\theta$ be the dual action of $\sigma^{\varphi}|_G$ on $N$. Then by Takesaki duality, $N$ is isomorphic to the fixed point subalgebra of the QWEP algebra $N\rtimes_{\theta}\widehat{G}\cong M\overline{\otimes}\mathbb{B}(\ell^2(G))$ by the dual action of $\theta$. Since $\widehat{\widehat{G}}=G$ is amenable, there is a conditional expectation of $N\rtimes_{\theta}\widehat{G}$ onto $N$, so $N$ has QWEP as well. Again by the existence of $\varepsilon_n$, each $N_n$ has QWEP too. Since $N_n$ is of finite type, by \cite[Theorem 4.1]{Kirchberg}, there is an embedding $i_n\colon N_n\to R^{\omega}$ (and a normal faithful conditional expectation $\varepsilon_n^0\colon R^{\omega}\to i_n(N_n)$, which automatically exists).
Note also that there is an embedding $i$ of $R$ into $R\overline{\otimes}R_{\infty}\cong R_{\infty}$ and there is a normal faithful conditional expectation $\varepsilon\colon R\overline{\otimes}R_{\infty}\to R$ given by a right-slice map by some $\varphi\in S_{\rm{nf}}(R_{\infty})$. Let $\tau$ be the unique tracial state on $R$, and let $\psi:=\tau\circ \varepsilon$. Then for any $(x_n)_n\in \ell^{\infty}(\mathbb{N},R)$, $(i(x_n))_n$ is in $\mathcal{M}^{\omega}(R_{\infty})$ (because $i(x_n)\in (R_{\infty})_{\psi})$, whence we have an embedding $i^{\omega}\colon R^{\omega}\to R_{\infty}^{\omega}, (x_n)^{\omega}\mapsto (i(x_n))^{\omega}$, and a normal faithful conditional expectation $\varepsilon^{\omega}\colon R_{\infty}^{\omega}\to R_{\infty}$ given by $(x_n)^{\omega}\mapsto (\varepsilon (x_n))^{\omega}$. Therefore there is an embedding $i_n'=i^{\omega}\circ i_n\colon N_n\to R_{\infty}^{\omega}$ and a normal faithful conditional expectation $\varepsilon_n'=i^{\omega}\circ \varepsilon_n^0\circ (i^{\omega})^{-1}\circ \varepsilon^{\omega}\colon R_{\infty}^{\omega}\to i_n'(N_n)$. Then by Theorem \ref{UW4.1.4}, $N_n\in \overline{\mathcal{F}_{\rm{inj}}}$ holds. Now, as each $N_n$ is in ${\rm{SA}}_{\widehat{\varphi}}(N)$ and $\varepsilon_n$ is the $\sigma^{\widehat{\varphi}}$-preserving conditional expectation, the argument before \cite[Corollary 2.12]{HW1} gives $N_n\to N$ in vN$(H)$. This shows that $N\in \overline{\mathcal{F}_{\rm{inj}}}$. Then by Theorem \ref{UW4.1.4}, there is an embedding $j\colon N\to R_{\infty}^{\omega}$ with a normal faithful conditional expectation $\varepsilon'\colon R_{\infty}^{\omega}\to j(N)$. Therefore we obtain an embedding $i\colon M\stackrel{\pi}{\to}N\stackrel{j}{\to}R_{\infty}^{\omega}$ together with a normal faithful conditional expectation $R_{\infty}^{\omega}\stackrel{\varepsilon'}{\to}j(N)\stackrel{j\circ \varepsilon \circ j^{-1}}{\to}i(M)$. Again by Theorem \ref{UW4.1.4}, $M\in \overline{\mathcal{F}_{\rm{inj}}}$ holds. 
\end{proof}
\begin{remark}
Although QWEP conjecture has not been settled, we remark that if there is at least one $M\in {\rm{vN}}(H)$ without QWEP, then the set $\text{vN}(H)^{\neg \text{QWEP}}$ of those $M\in \text{vN}(H)$ without QWEP, is an open and dense subset of $\text{vN}(H)$. Indeed, by Theorem \ref{thm: QWEP iff closure of injective factors}, $\text{vN}(H)^{\neg \text{QWEP}}$ is an open set. Suppose there is one $M\in \text{vN}(H)$ without QWEP, and let $N\in \text{vN}(H)$. Then by choosing $v_0\in \mathcal{U}(H\otimes H,H)$, \cite[Lemma 2.4]{HW2}, there is $(u_n)_n\in \mathcal{U}(H\otimes H)$ such that $v_0u_n^*(N\overline{\otimes}M)u_nv_0^*\to N$ in vN$(H)$. Since $v_0u_n^*(N\overline{\otimes}M)u_nv_0^*$ fails QWEP for each $n\in \mathbb{N}$, $N$ is in the closure of $\text{vN}(H)^{\neg \text{QWEP}}$. 
\end{remark}
\begin{theorem}\label{thm: QWEP via approximation of standard form}
Let $M$ be a von Neumann algebra with separable predual, and let $\mathcal{P}_M^{\natural}$ be the natural cone in the standard form of $M$. The following conditions are equivalent.
\begin{itemize}
\item[{\rm{(1)}}] $M$ has {\rm{QWEP}}.
\item[{\rm{(2)}}] For any $n\in \mathbb{N}$, $\xi_1,\cdots,\xi_n\in \mathcal{P}_M^{\natural}$, and $\varepsilon>0$, there exists $k\in \mathbb{N}$ and $a_1,\cdots,a_n\in M_k(\mathbb{C})_+$ such that 
\[|\nai{\xi_i}{\xi_j}-{\rm{tr}}_k(a_ia_j)|<\varepsilon\ \ \ \ \ \ (i,j=1,\cdots,n).\]
Here, ${\rm{tr}}_k$ is the normalized trace on $M_k(\mathbb{C})$.
\end{itemize}
\end{theorem}
Note that if $M$ is a II$_1$ factor, then $\mathcal{P}_M^{\natural}=\overline{M_+\xi_{\tau}}$. Therefore in this case, the above theorem is analogous to the following Kirchberg's result \cite[Proposition 4.6]{Kirchberg} when unitaries in his results are replaced by positive elements. 
\begin{theorem}[Kirchberg]\label{thm: Kirchberg approximation}
Let $M$ be a finite von Neumann algebra with separable predual with a normal faithful tracial state $\tau_M$. Then the following conditions are equivalent:
\begin{list}{}{}
\item[{\rm{(1)}}]
There exists an embedding $i\colon M\to R^{\omega}$ with $\tau^{\omega}(i(a))=\tau_M(a)\ (a\in M)$.
\item[{\rm{(2)}}] For every $\varepsilon>0, n\in \mathbb{N}$ and $u_1,\cdots,u_n\in \mathcal{U}(M)$, there is $k\in \mathbb{N}$ and $v_1,\cdots,v_n\in \mathcal{U}(M_k(\mathbb{C}))$ such that 
\[|\tau_M(u_i^*u_j)-{\rm{tr}}_k(v_i^*v_j)|<\varepsilon,\ \ \ |\tau_M(u_i)-{\rm{tr}}_k(v_i)|<\varepsilon\]
for $i,j=1,\cdots,n$. 
\end{list}  
\end{theorem}
The next lemma plays the key role in the proof of Theorem \ref{thm: QWEP via approximation of standard form}
\begin{lemma}\label{lem: keylemma}
Let $(M,H_M,J_M,\mathcal{P}_M^{\natural}), (N,H_N,J_N,\mathcal{P}_N^{\natural})$ be standard forms. Assume that $M$ has {\rm{QWEP}}, $N$ is $\sigma$-finite and there is an linear isometry $\rho\colon H_N\to H_M$ satisfying $\rho(\mathcal{P}_N^{\natural})\subset \mathcal{P}_M^{\natural}$. Then $N$ has {\rm{QWEP}}.
\end{lemma}
To prove the lemma, we use following classical results on the natural cone. 
\begin{lemma}[Araki]\label{lem: order isomorphism of positive cone} Let $(M,H,J,\mathcal{P}_M^{\natural})$ be a standard form with $M$ $\sigma$-finite. Let $\varphi\in S_{\rm{nf}}(M)$ with $\varphi=\omega_{\xi_{\varphi}}$, $\xi_{\varphi}\in \mathcal{P}_M^{\natural}$. Then $\Phi\colon M_{\rm{sa}}\ni a\mapsto \Delta_{\varphi}^{\frac{1}{4}}a\xi_{\varphi}\in H_{{\rm{sa}}}$ induces an order isomorphism between $\{a\in M_{\rm{sa}};\ -\alpha 1\le a\le \alpha 1\}$ and $\{\xi\in H_{\rm{sa}}; -\alpha \xi_{\varphi}\le \xi\le \alpha \xi_{\varphi}\}$ for each $\alpha>0$.  
\end{lemma}
\begin{proof}
For $a\in M$, it is easy to see that
\[a\in M_+\Leftrightarrow \Delta_{\varphi}^{\frac{1}{4}}a\xi_{\varphi}\in \mathcal{P}_M^{\natural}.\]
This implies that $\Phi$ is an order isomorphism from $M_{\rm{sa}}$ onto $\Phi(M_{\rm{sa}})$. By \cite[Theorem 3.8 (8)]{Araki}, $\Phi$ maps $\{x\in M;\ 0\le x\le \alpha 1\}$ onto $\{\eta\in \mathcal{P}_M^{\natural};\ 0\le \eta\le \alpha \xi_{\varphi}\}$. From this the claim easily follows. 
\end{proof}

\begin{lemma}\label{lem: appendix about natural cone}
Let $(M,H,J,\mathcal{P}_M^{\natural})$ be a standard form. For each $\xi\in \mathcal{P}_M^{\natural}$, denote by $e(\xi)$ the support projection of $\omega_{\xi}$. Let $q:=e(\xi)Je(\xi)J$. Then the following holds.
\begin{list}{}{}
\item[{\rm{(1) (Araki)}}] $\xi\perp \eta\Leftrightarrow e(\xi)\perp e(\eta)$. 
\item[{\rm{(2)}}] For $\xi,\eta\in \mathcal{P}_M^{\natural}$, $e(\xi)\le e(\eta)$ if and only if $\eta\perp \zeta\Rightarrow \xi\perp \zeta$ holds for every $\zeta\in \mathcal{P}_M^{\natural}$.
\item[{\rm{(3)}}] $q(\mathcal{P}_M^{\natural})=(\{\xi\}^{\perp}\cap \mathcal{P}_M^{\natural})^{\perp}\cap \mathcal{P}_M^{\natural}$.
\end{list}
\end{lemma}

\begin{proof} (1) $(\Rightarrow )$ is \cite[Theorem 4, (7)]{Araki}, while $(\Leftarrow)$ is clear. (2) and (3) have been known. We include a proof from second named author's master thesis for reader's convenience. \\
(2) $(\Rightarrow )$ follows from (1). 
$(\Leftarrow)$ Let $\{\xi_i\}_{i\in I}$ be a family of elements in $\mathcal{P}_M^{\natural}$ such that $\{e(\xi_i)\}_{i\in I}$ are pairwise orthogonal and $\sum_{i\in I}e(\xi_i)=1-e(\eta)$ (it can be proved that every $\sigma$-finite projection $p\in M$ is of the form $p=e(\eta)$ for some $\eta\in \mathcal{P}_M^{\natural}$). Hence $\xi_i\perp \eta$, which implies  that $\xi_i\perp \xi$ for all $i\in I$. By (1), $e(\xi_i)\perp e(\xi)\ (i\in I)$ holds. Therefore 
\[e(\xi)\le 1-\sum_{i\in I}e(\xi_i)=e(\eta).\]
\\
(3) By (2), it holds that for $\eta\in \mathcal{P}_M^{\natural}$,
\eqa{
e(\eta)\le e(\xi)&\Leftrightarrow \forall \zeta\in \mathcal{P}_M^{\natural}\ (\zeta\perp \xi\Rightarrow \zeta\perp \eta)\\
&\Leftrightarrow \eta\in (\{\xi\}^{\perp}\cap \mathcal{P}_M^{\natural})^{\perp}\cap \mathcal{P}_M^{\natural}.
}
If $e(\eta)\le e(\xi)$, then $e(\xi)\eta=\eta$, so $q\eta=Je(\xi)Je(\xi)\eta=Je(\xi)\eta=J\eta=\eta$. Thus, $\eta=q(\eta)\in q(\mathcal{P}_M^{\natural})$. On the other hand, if $\eta\in q(\mathcal{P}_M^{\natural})$, then clearly $e(\eta)\le e(\xi)$ holds. This shows that $q(\mathcal{P}_M^{\natural})=(\{\xi\}^{\perp}\cap \mathcal{P}_M^{\natural})^{\perp}\cap \mathcal{P}_M^{\natural}$.
\end{proof}

The next result is \cite[Corollary 5.3]{Ozawa}.
\begin{lemma}[\cite{Kirchberg,Ozawa}] \label{lem: Kirchberg reduction}Let $M$ be a von Neumann algebra, and if there is a {\rm{C}}$^*$-algebra $A$ with {\rm{QWEP}} and a contractive linear map $\varphi\colon A\to M$ such that $\varphi(\overline{Ball}(A))$ is ultraweakly dense in $\overline{Ball}(M)$, then $M$ has {\rm{QWEP}}. Here, $\overline{Ball}(B)$ is the closed unit ball of a {\rm{C}}$^*$-algebra $B$.
\end{lemma}

\begin{proof}[Proof of Lemma \ref{lem: keylemma}]
Put $e\colon =\rho \rho^*\in \mathbb{B}(H_M)$. \\ \\
\textbf{Claim 1}. $\rho (\mathcal{P}_N^{\natural})=e(\mathcal{P}_M^{\natural})=\mathcal{P}_M^{\natural}\cap e(H_M)$, and $\rho J_N=J_M\rho$.\\
First, $\rho (\mathcal{P}_N^{\natural})=\rho \rho^* \rho (\mathcal{P}_N^{\natural})\subset e(\mathcal{P}_M^{\natural})$ holds. To prove $e(\mathcal{P}_M^{\natural})\subset \rho (\mathcal{P}_N^{\natural})$, we show $\rho^*(\mathcal{P}_M^{\natural})\subset \mathcal{P}_N^{\natural}$. Let $\xi\in \mathcal{P}_M^{\natural}$. Then for each $\eta\in \mathcal{P}_N^{\natural}$,
we have $\nai{\rho^*(\xi)}{\eta}=\nai{\xi}{\rho (\eta)}\ge 0$, because $\rho (\eta)\in \mathcal{P}_M^{\natural}$, whence $\rho^*(\xi)\in (\mathcal{P}_N^{\natural})^0=\mathcal{P}_N^{\natural}$ and $\rho^*(\mathcal{P}_M^{\natural})\subset \mathcal{P}_N^{\natural}$ holds. Therefore for $\xi \in \mathcal{P}_M^{\natural}$, $e\xi=\rho (\rho^*\xi)\in \rho \mathcal{P}_N^{\natural}$. This proves $\rho (\mathcal{P}_N^{\natural})=e\mathcal{P}_M^{\natural}$. Note that this also shows that $e(\mathcal{P}_M^{\natural})\subset \mathcal{P}_M^{\natural}$. Therefore $e(\mathcal{P}_M^{\natural})=\mathcal{P}_M^{\natural}\cap e(H_M)$ holds too. To see that last equality, recall that any $\xi\in H_N$ can be written as 
$\xi=(\xi_1-\xi_2)+i(\xi_3-\xi_4)$, where $\xi_j\in \mathcal{P}_N^{\natural}$ and $J_N\xi_j=\xi_j\ (1\le j\le 4)$. From this it holds that
\eqa{
\rho J_N\xi &=\rho \left \{(\xi_1-\xi_2)-i(\xi_3-\xi_4)\right \}=J_M\left \{(\rho \xi_1-\rho \xi_2)+i(\rho \xi_3-\rho \xi_4)\right \}\\
&=J_M\rho \xi.
}
Now fix $\psi \in S_{\rm{nf}}(N)$. Then $\psi=\omega_{\xi_{\psi}}, \xi_{\psi}\in \mathcal{P}_N^{\natural}$. Let $\varphi:=\omega_{\eta}, \eta:=\rho(\xi_{\psi})\in \mathcal{P}_M^{\natural}$, and $p:=\text{supp}(\varphi)\in M.$\\ \\ 
\textbf{Claim 2.} $q:=pJ_MpJ_M\ge e$.\\
Since $H_M$ is spanned by $\mathcal{P}_M^{\natural}$, it suffices to show that $q(\mathcal{P}_M^{\natural})\supset e(\mathcal{P}_M^{\natural})=\rho(\mathcal{P}_N^{\natural})$. 
Let $\xi\in \mathcal{P}_N^{\natural}$. Let $\zeta \in \{\rho(\xi_{\psi})\}^{\perp}\cap \mathcal{P}_M^{\natural}$. 
Then $0=\nai{\zeta}{\rho(\xi_{\psi})}=\nai{\rho^*(\zeta)}{\xi_{\psi}}$. 
Since $\psi=\omega_{\xi_{\psi}}$ is faithful, by Lemma \ref{lem: appendix about natural cone} (2), $\zeta'\perp \xi_{\psi}\Rightarrow \zeta' \perp \xi$ holds for every $\zeta'\in \mathcal{P}_N^{\natural}$. 
Therefore $\rho^*(\zeta)\perp \xi\Leftrightarrow \zeta\perp \rho(\xi)$ holds. This implies, by Lemma \ref{lem: appendix about natural cone} (3), that 
\[\rho(\xi)\in (\{\rho(\xi_{\psi})\}^{\perp}\cap \mathcal{P}_M^{\natural})^{\perp}\cap \mathcal{P}_M^{\natural}=q(\mathcal{P}_M^{\natural}).\]
Therefore $e(\mathcal{P}_M^{\natural})=\rho(\mathcal{P}_N^{\natural})\subset q(\mathcal{P}_M^{\natural})$, and $e\le q$ holds. 

Since $M$ has QWEP, so does $qMq$. Therefore we may replace $(M,H_M,J_M,\mathcal{P}_M^{\natural})$ by $(qMq, qH_M,J_M|_{qH_M}, q(\mathcal{P}_M^{\natural}))$ and assume that $q=1$, and $\varphi$ is faithful on $M$ with $\xi_{\varphi}=\rho(\xi_{\psi})\in \mathcal{P}_M^{\natural}$. In this case we may identify $\mathcal{P}_M^{\natural}=\overline{\Delta_{\varphi}^{\frac{1}{4}}M_+\xi_{\varphi}},\ \mathcal{P}_N^{\natural}=\overline{\Delta_{\psi}^{\frac{1}{4}}N_+\xi_{\psi}}$.\\ \\
\textbf{Claim 3.} There are unital positive maps $i\colon N\to M,\ \varepsilon\colon M\to N$ such that the following diagram commutes.
\[
\xymatrix{
N_+\ar@{}[rd]|{\circlearrowright}\ar[d]_{\beta}\ar[r]^i & M_+\ar@{}[rd]|{\circlearrowright}\ar[d]^{\alpha}\ar[r]^{\varepsilon}& N_+\ar[d]^{\beta}\\
\mathcal{P}_N^{\natural}\ar[r]_{\rho}&\mathcal{P}_M^{\natural}\ar[r]_{\rho^*}&\mathcal{P}_N^{\natural}
}
\]
where $\alpha\colon M_+\ni x\mapsto \Delta_{\varphi}^{\frac{1}{4}}x\xi_{\varphi}\in \mathcal{P}_M^{\natural},\ \beta\colon N_+\ni y\mapsto \Delta_{\psi}^{\frac{1}{4}}y\xi_{\psi}\in \mathcal{P}_N^{\natural}$. 
To construct $i$, let $y\in N_{\rm{sa}}$. Then as $-\|y\|1\le y\le \|y\|1$ in $N$, Lemma \ref{lem: order isomorphism of positive cone} asserts that $-\|y\| \xi_{\psi}\le \Delta_{\psi}^{\frac{1}{4}}y\xi_{\psi}\le \|y\|\xi_{\psi}$ in $H_{N,{\rm{sa}}}=\{\xi\in H_N;\ J_N\xi=\xi\}$. Therefore as $\rho$ preserves the order, $-\rho(\|y\|\xi_{\psi})=-\|y\|\xi_{\varphi}\le \rho(\Delta_{\psi}^{\frac{1}{4}}y\xi_{\psi})\le \|y\|\xi_{\varphi}$ in $H_{M,{\rm{sa}}}$. Again by Lemma \ref{lem: order isomorphism of positive cone} applied to $H_{M,{\rm{sa}}}$, there exists unique $\theta (y)\in M_{\rm{sa}}$ with $-\|y\|1\le \theta (y)\le \|y\|1$ in $M$, such that 
\begin{equation}
\rho(\Delta_{\psi}^{\frac{1}{4}}y\xi_{\psi})=\Delta_{\varphi}^{\frac{1}{4}}\theta(y)\xi_{\varphi}.\label{eq: definition of theta}
\end{equation} 
Then by the form of (\ref{eq: definition of theta}), the linearity of $\rho$ asserts that $\theta\colon N_{\rm{sa}}\to M_{\rm{sa}}$ is a real linear map. It is also clear by the construction that $\theta(N_+)\subset M_+$ and $\theta(1)=1$. Therefore $\theta$ can be extended to a complex linear unital positive map $i\colon N\to M$ by
\[i(y):=\theta(y_1)+i\theta(y_2),\ \ \ \ \ \ y=y_1+iy_2,\ \ \ y_1,y_2\in N_{\rm{sa}}.\]
Similarly, we define a unital positive map $\varepsilon\colon M\to N$ by
\[\varepsilon(x):=\pi(x_1)+i\pi(x_2),\ \ \ \ \ \ x=x_1+ix_2,\ \ \ x_1,x_2\in M_{\rm{sa}}\]
 where for $x\in M_{\rm{sa}}$, $\pi(x)\in N_{\rm{sa}}$ is the unique element satisfying 
\[\rho^*(\Delta_{\varphi}^{\frac{1}{4}}x\xi_{\varphi})=\Delta_{\psi}^{\frac{1}{4}}\pi(x)\xi_{\psi}.\]
Note that since $\rho^*(\mathcal{P}_M^{\natural})\subset \mathcal{P}_N^{\natural}$, $\pi(x)$ is well-defined by Lemma \ref{lem: order isomorphism of positive cone}. By construction, the above diagram commutes. Finally, let $y\in N_+$. Then by construction, $\rho(\Delta_{\psi}^{\frac{1}{4}}y\xi_{\psi})=\Delta_{\varphi}^{\frac{1}{4}}i(y)\xi_{\varphi}$ holds. Then we apply $\rho^*$ to obtain $\Delta_{\psi}^{\frac{1}{4}}y\xi_{\psi}=\rho^*(\Delta_{\varphi}^{\frac{1}{4}}i(y)\xi_{\varphi})$. Therefore by the construction of $\varepsilon$, $\varepsilon\circ i(y)=y$ holds. Therefore by linearity, $\varepsilon\circ i={\rm{id}}_N$. In particular, $\varepsilon$ maps the closed unit ball of $M$ onto the closed unit ball of $N$. Since $\varepsilon$ is positive hence contractive, and since $M$ has QWEP, $N$ also has QWEP by Lemma \ref{lem: Kirchberg reduction}. 
\end{proof}
\begin{proof}[Proof of Theorem \ref{thm: QWEP via approximation of standard form}]
(1)$\Rightarrow$(2) Assume that $M$ has QWEP. then by Theorem \ref{thm: QWEP iff closure of injective factors}, $M\in \overline{\mathcal{F}_{\rm{inj}}}$ holds. 
%By (the proof of) \cite[Corollary 2.11]{HW2}, the set $\mathcal{F}_{{\rm{I}}_{{\rm{fin}}}}$ factors is dense in $\mathcal{F}_{\rm{inj}}$. 
By Theorem \ref{UW4.1.4} (i)$\Leftrightarrow $(iv), there exist
$\{k_n\}_{n=1}^{\infty}\subset \mathbb{N}$, a sequence $\varphi_n\in S_{\rm{nf}}(M_{k_n}(\mathbb{C}))$ such that $M$ admits an embedding $i\colon M\to (M_{k_n}(\mathbb{C}),\varphi_n)^{\omega}$ and a normal faithful conditional expectation $\varepsilon\colon (M_{k_n}(\mathbb{C}),\varphi_n)^{\omega}\to i(M)$. 
%koko
Let $(M_{k_n}(\mathbb{C}),K_n,J_n,\mathcal{P}_n^{\natural})$ be the standard form constructed from the GNS representation of $\varphi_n$ for each $n\in \mathbb{N}$. Then by \cite[Theorem 3.18]{AndoHaagerup}, $(\prod^{\omega}M_{k_n}(\mathbb{C}),K_{\omega},J_{\omega},\mathcal{P}_{\omega}^{\natural})$ is a standard form, where $(K_{\omega},J_{\omega},\mathcal{P}_{\omega}^{\natural})$ is the ultraproduct of $(K_n,J_n,\mathcal{P}_n^{\natural})_n$ (see \cite[Theorem 3.18]{AndoHaagerup}). Let $p=\text{supp}((\varphi_n)_{\omega})\in \prod^{\omega}M_{k_n}(\mathbb{C})$, and $q:=pJ_{\omega}pJ_{\omega}$. Then by \cite[Corollaries 3.27 and 3.28]{AndoHaagerup}, $(q(\prod^{\omega}M_{k_n}(\mathbb{C}))q,qK_{\omega},J_{\omega}|_{qK_{\omega}},q\mathcal{P}_{\omega}^{\natural})$ can be regarded as a standard representation of $\widetilde{M}=(M_{k_n}(\mathbb{C}),\varphi_n)^{\omega}$. Recall also that the natural cone $\mathcal{P}_n^{\natural}$ can be also regarded as the one constructed from the GNS representation of the trace $\text{tr}_{k_n}$, we see that $\mathcal{P}_n^{\natural}$ satisfies conditon (2) for each $n$. Then $\mathcal{P}_{\omega}^{\natural}$ also satisfies condition (2): let $\varepsilon>0, m\in \mathbb{N}$ and $\xi_1=(\xi_{1,n})_{\omega},\cdots, \xi_m=(\xi_{m,n})_{\omega}\in \mathcal{P}_{\omega}^{\natural}$ be given, where $\xi_{i,n}\in \mathcal{P}_n^{\natural}\ (1\le i\le m,n\in \mathbb{N})$. Then we have 
\[\nai{\xi_i}{\xi_j}=\lim_{n\to \omega}\nai{\xi_{i,n}}{\xi_{j,n}}=\lim_{n\to \omega}\text{tr}_{k_n}(\xi_{i,n}\xi_{j,n})\ (1\le i,j\le m),\]
so that 
\[J:=\{n\in \mathbb{N};\ |\nai{\xi_i}{\xi_j}-\text{tr}_{k_n}(\xi_{i,n}\xi_{j,n})|<\varepsilon,\ 1\le i,j\le m\}\in \omega.\]
Choose $n_0\in J$, and put $\eta_i:=\xi_{i,n_0}\ (1\le i\le m)$. Then the inequality in (2) is satisfied.\ But as $q(\mathcal{P}_{\omega}^{\natural})\subset \mathcal{P}_{\omega}^{\natural}$, $q(\mathcal{P}_{\omega}^{\natural})$ also satisfies (2). Furthermore, $i(M)\subset \widetilde{M}$ and we have a normal faithful conditional expectation $\varepsilon\colon \widetilde{M}\to i(M)$. Thus choosing $\psi\in S_{\rm{nf}}(i(M))$ and letting $\widetilde{\psi}:=\psi\circ \varepsilon$, we obtain 
\[\mathcal{P}_M^{\natural}\cong \mathcal{P}_{i(M)}^{\natural}=\overline{\Delta_{\psi}^{\frac{1}{4}}i(M)_+\xi_{\psi}}\subset \mathcal{P}_{\widetilde{M}}^{\natural}=
\overline{\Delta_{\widetilde{\psi}}^{\frac{1}{4}}\widetilde{M}_+\xi_{\widetilde{\psi}}}\cong q(\mathcal{P}_{\omega}^{\natural}).\]
Here, $\cong$ means that the one natural cone is mapped to the other by a unitary implementing the isomorphism of a standard form. Therefore $\mathcal{P}_M^{\natural}$ satisfies condition (2).\\
(2)$\Rightarrow$(1) Assume $M$ satisfies (2). Let $\{\xi_n\}_{n=1}^{\infty}\subset \mathcal{P}_M^{\natural}$ be a dense sequence. Using condition (2), for each $n\in \mathbb{N}$, choose $k_n\in \mathbb{N}$ and $a_1^{(n)},\cdots,a_n^{(n)}\in M_{k_n}(\mathbb{C})_+$ such that 
\[|\nai{\xi_i}{\xi_j}-\text{tr}_{k_n}(a_i^{(n)}a_j^{(n)})|<\frac{1}{n},\ \ \ \ 1\le i,j\le n.\]
Put 
\[ \xi_i:=\begin{cases}a_i^{(n)}\xi_{\text{tr}_{k_n}}\in \mathcal{P}_n^{\natural}:=\mathcal{P}_{M_{k_n}(\mathbb{C})}^{\natural} & (n\ge i)\\
\ \ \ \ \ \ \ \ \ \ \ \ \ \ 0 & (n<i)
\end{cases},
\]
and $\widetilde{\xi}_i:=(\xi_i^{(n)})_{\omega}\in \mathcal{P}_{\omega}^{\natural}:=(\mathcal{P}_n^{\natural})_{\omega}$. Then
\[\nai{\widetilde{\xi}_i}{\widetilde{\xi}_j}=\lim_{n\to \omega}\nai{\xi_i^{(n)}}{\xi_j^{(n)}}=\nai{\xi_i}{\xi_j},\ \ \ \ i,j\in \mathbb{N}.\]
Fix one $\varphi \in S_{{\rm{nf}}}(M)$. Then on the dense subspace $K_0:=\text{span}\{\xi_i;i\ge 1\}$ of $L^2(M,\varphi)$, define $\rho_0\colon K_0\to H_{\omega}$ by 
\[\rho_0\left (\sum_{j=1}^k\lambda_{i_j}\xi_{i_j}\right ):=\sum_{j=1}^{k}\lambda_{i_j}\widetilde{\xi}_{i_j},\ \ \ \ \lambda_{i_j}\in \mathbb{C}\ (1\le i\le k).\]
 Then $\rho_0$ is uniquely extended to a linear isometry $\rho\colon L^2(M,\varphi)\to H_{\omega}$ such that $\rho(\mathcal{P}_M^{\natural})\subset \mathcal{P}_{\omega}^{\natural}$. Since $\widetilde{M}:=\prod^{\omega}M_{k_n}(\mathbb{C})$ has QWEP with natural cone $\mathcal{P}_{\omega}^{\natural}$, $M$ has QWEP thanks to Lemma \ref{lem: keylemma}.   
\end{proof}
As a corollary to Theorems \ref{UW4.1.4}, \ref{thm: QWEP iff closure of injective factors} and \ref{thm: QWEP via approximation of standard form}, we have the following:
\begin{theorem}\label{UW4.1.5}
Let $H$ be a separable infinite-dimensional Hilbert space, and let $0<\lambda<1$. The following conditions are equivalent:
\begin{list}{}{}
\item[{\rm{(a)}}] $\mathcal{F}_{\rm{inj}}$ is dense in {\rm{vN}}$(H)$.
\item[{\rm{(b)}}] For any $N\in {\rm{vN}}(H)$, there is an embedding $i\colon N\to R_{\infty}^{\omega}$ and a normal faithful conditional expectation $\varepsilon\colon R_{\infty}^{\omega}\to i(N)$.
\item[{\rm{(c)}}] For any $N\in {\rm{vN}}(H)$, there is an embedding $i\colon N\to R_{\lambda}^{\omega}$ and a normal faithful
 conditional expectation $\varepsilon\colon R_{\lambda}^{\omega}\to i(N)$.
\item[{\rm{(d)}}] For any $N\in \mathcal{F}_{{\rm{II}}_1}$, there is an embedding $i\colon N\to R^{\omega}$. 
\item[{\rm{(e)}}] Every $N\in {\rm{vN}}(H)$ has {\rm{QWEP}}.
\item[{\rm{(f)}}] For any $N\in {\rm{vN}}(H)$, there is $\{n_k\}_{k=1}^{\infty}\subset \mathbb{N}$, a sequence $\varphi_k\in S_{\rm{nf}}(M_{n_k}(\mathbb{C}))$ such that $N$ admits an embedding $i\colon N\to (M_{n_k}(\mathbb{C}),\varphi_k)^{\omega}$ and a normal faithful conditional expectation $\varepsilon\colon (M_{n_k}(\mathbb{C}),\varphi_k)^{\omega}\to i(N)$.
\item[{\rm{(g)}}] For any $N\in {\rm{vN}}(H)$, $\varepsilon>0$, $n\in \mathbb{N}$ and $\xi_1,\cdots,\xi_n\in \mathcal{P}_N^{\natural}$, there exist $k\in \mathbb{N}$ and $a_1,\cdots,a_n\in M_k(\mathbb{C})_+$ such that 
\[|\nai{\xi_i}{\xi_j}-{\rm{tr}}_k(a_ia_j)|<\varepsilon\ \ \ \ \ \ (1\le i,j\le n)\]
holds, where $\mathcal{P}_N^{\natural}$ is the natural cone in the standard form of $N$. 
\end{list}
\end{theorem}

\section*{Acknowledgements}
The authors thank Masaki Izumi for useful comments and informations about literature. HA is supported by EPDI/JSPS/IH\'ES Fellowship. UH is supported by ERC Advanced Grant No. OAFPG 27731 and the Danish National Research Foundation through the Center for Symmetry and Deformation.

%\begin{lemma}Let $F=\{F_t\}_{t\in \mathbb{R}}$ be a properly ergodic flow on a standard probability space. 
%Then the set $\mathcal{F}_{{\rm{III}}_0}^{{\rm{inj}},F}$ of all injective type {\rm{III}}$_0$ factors on $H$ with flow of weights isomorphic to $F$, is dense in $\mathcal{F}^{{\rm{inj}}}$.  
%\end{lemma}

Hiroshi Ando\\
Institut des Hautes \'Etudes Scientifiques,\\
Le Bois-Marie 35, route de Chartres,\\
 91440 Bures-sur-Yvette, France\\
ando@ihes.fr 
\\ \\
Uffe Haagerup\\
Department of Mathematical Sciences,\\
University of Copenhagen\\
Universitetsparken 5,\\
2100 Copenhagen \O, Denmark\\
haagerup@math.ku.dk\\ \\
Carl Winsl\o w\\
Department of Science Education,\\
University of Copenhagen\\
\O ster Voldgade 3,\\ 
 1350 Copenhagen C, Denmark\\
 winslow@ind.ku.dk
\end{document}